\documentclass[12pt]{article}

\def\nfrac#1#2{{\textstyle\frac{#1}{#2}}}

\usepackage{amsmath,amssymb,amsthm}
\usepackage{dsfont,nicefrac}
\usepackage{graphicx}

\newtheorem{theorem}{Theorem}[section]
\newtheorem{lemma}[theorem]{Lemma}
\newtheorem{remark}[theorem]{Remark}

\newtheorem{corollary}[theorem]{Corollary}

\newtheorem{definition}[theorem]{Definition}

\def\nfrac#1#2{{\textstyle\frac{#1}{#2}}}
\def\dfrac#1#2{\lower0.15ex\hbox{\large$\frac{#1}{#2}$}}

\makeatletter

\renewcommand{\p@enumii}{}

\renewcommand{\p@enumiii}{}
\makeatother

\title{Structure and eigenvalues of heat-bath Markov chains}

\author{
Martin Dyer\thanks{Research supported by EPSRC Grant EP/J006300/1
and by a UNSW Faculty of Science Visiting Researcher Fellowship.}\\
\small School of Computing\\[-0.5ex]
\small University of Leeds\\[-0.5ex]
\small Leeds LS2 9JT, UK\\[-0.5ex]
\small \tt m.e.dyer@leeds.ac.uk\\
\and
Catherine Greenhill\thanks{
Research supported by Australian Research Council.}\\
\small School of Mathematics and Statistics\\[-0.5ex]
\small The University of New South Wales\\[-0.5ex]
\small Sydney NSW 2052, Australia\\[-0.5ex]
\small \tt csg@unsw.edu.au\\
\and
Mario Ullrich\thanks{
Research partially supported by ERC Advanced Grant PTRELSS and
DFG grant DFG GRK 1523.}\\
\small Mathematisches Institut\\[-0.5ex]
\small Friedrich-Schiller-Universit\"at\\[-0.5ex]
\small 07743 Jena, Germany\\[-0.5ex]
\small \tt mario.ullrich@uni-jena.de
}

\date{10 April 2014\\[0.5ex]
\small Keywords: stochastic matrices, Markov chains, heat-bath, eigenvalues, positive semidefinite\\[0.5ex]
  MSC 2010: 15B51,60J10 }

\begin{document}

\maketitle

\begin{abstract}
We prove that heat-bath chains (which we define in a general setting)
have no negative eigenvalues.  
Two applications of this result are presented: one to single-site heat-bath
chains for spin
systems and one to a heat-bath Markov chain for sampling contingency tables.
Some implications of our main result
for the analysis of the mixing time of heat-bath
Markov chains are discussed.
We also prove an alternative characterisation of heat-bath chains,
and consider possible generalisations.
\end{abstract}

\section{Definitions and our first result}\label{s:result}

Suppose that $\Omega$ is a finite set 
and let $\pi:\Omega\to (0,1]$ be a probability distribution on $\Omega$.
Let $\mathcal{L}$ be a nonempty finite index set 
and let $L= |\mathcal{L}|$.  
Suppose that for all
$x\in\Omega$ and $a\in\mathcal{L}$ we have a subset 
$\Omega_{x,a}$ of $\Omega$ such that
\begin{itemize}
\item[(I)]
$x\in \Omega_{x,a}$ for all $x\in\Omega$ and $a\in\mathcal{L}$,
and 
\item[(II)] for each $a\in \mathcal{L}$, the set
$\{ \Omega_{x,a} : x\in\Omega\}$ forms a partition of $\Omega$.
\end{itemize}
For $a\in\mathcal{L}$, define the $|\Omega|\times |\Omega|$ matrix
$P_a$ 
(with rows and columns indexed by $\Omega$) by
\begin{equation}
\label{Pa} P_a(x,y) = \frac{\pi(y)}{\pi(\Omega_{x,a})}\, 
                    \mathds{1}(y\in\Omega_{x,a}).
\end{equation}
(Here $\mathds{1}(y\in\Omega_{x,a})$ equals 1 if $y\in\Omega_{x,a}$,
and equals 0 otherwise.)
Note that $P_a$ is well-defined for all $a\in\mathcal{L}$,
since $\pi$ is nonzero on all states and each set $\Omega_{x,a}$ is nonempty.

Now for a given probability distribution $\rho$ on $\mathcal{L}$,
let $P$ be the $|\Omega|\times|\Omega|$ matrix defined by
\begin{equation}
\label{general}
P =  \sum_{a\in \mathcal{L}} \rho(a)\,  P_a.  
\end{equation}

Since $P$ is a stochastic matrix, it defines a Markov chain
$\mathcal{M}$ on $\Omega$, determined uniquely by $\pi$, 
$\mathcal{L}$, $\rho$, and the sets $\Omega_{x,a}$.
A transition of $\mathcal{M}$ from current state $x\in\Omega$ is performed by
choosing an element $a\in\mathcal{L}$ according to the distribution
$\rho$, then sampling the next state $y$ from $\Omega_{x,a}$
with respect to the distribution $\pi$ restricted
to $\Omega_{x,a}$.   

\begin{definition}
\emph{
A Markov chain $\mathcal{M}$ on a finite state space $\Omega$
is a \emph{heat-bath chain} if its transition matrix $P$
satisfies (\ref{general})
with respect to 
some finite nonempty set $\mathcal{L}$ equipped with a probability
distribution $\rho$, some probability distribution $\pi:\Omega\to (0,1]$,  
and some sets $\Omega_{x,a}$ which satisfy (I), (II). 
Here the matrices $P_a$ in (\ref{general})
are defined by (\ref{Pa}).
} \hfill $\diamond$
\label{d:heat-bath}
\end{definition}

Note that conditions (I) and (II) imply that 
for all $x,y\in\Omega$ and $a\in\mathcal{L}$,
\begin{equation}
\label{condition}
\text{if \, $y\in\Omega_{x,a}$ \, then \, $\Omega_{x,a} = \Omega_{y,a}$}.
\end{equation}
Furthermore, when (\ref{general}) holds it follows that $\mathcal{M}$ is 
aperiodic
(since every state has a self-loop) and that $\mathcal{M}$ is
reversible with respect to $\pi$.  However, the chain $\mathcal{M}$
need not be irreducible.
(See~\cite{jerrum-book} for Markov chain definitions which are
not given here.)

Before proceeding, we indicate how our definition of heat-bath chains 
corresponds to the usual notion of heat-bath chains, in the
setting of graph colourings or the Potts model.  In such a chain, the 
state space is a subset of
$S^V$ for some finite sets $V$, $S$.  
To express this using our formulation, let $\mathcal{L}=\{a_1,\dots,a_L\}$ 
be the set of all those subsets $a_i\subset V$ which may be updated by a 
single transition of the chain, and, for $a\in\mathcal{L}$, 
let $\Omega_{x,a}=x_{V\setminus a}\times S^a$ be the set of all states which
can be obtained from $x\in\Omega$ by ``recolouring'' or reassigning the
values at elements of $a$. 
Here $x_W=(x_v)_{v\in W}$ denotes
the restriction of $x$ to $W$, for all $x\in S^V$ and $W\subseteq V$.
So $\Omega_{x,a}$ contains
all possibilities for the next state of the chain, given that $x$
is the current state and that $a\in\mathcal{L}$ was chosen by the
transition procedure.
See also the examples presented in Section~\ref{s:applications}.

\begin{lemma}
Suppose that $\mathcal{M}$ is a heat-bath chain, in the sense
of Definition~\ref{d:heat-bath}.
Then $\mathcal{M}$ has no negative eigenvalues.
\label{heat-bath}
\end{lemma}

\begin{proof}
By definition of $P_a$ we know that $P_a(x,y)=0$ if $y\not\in 
\Omega_{x,a}$.  Furthermore (\ref{condition}) implies that
if $z\in\Omega_{x,a}$ then $P_a(x,y) = P_a(z,y)$ for all
$y\in\Omega$.  Therefore, for all $x,y\in\Omega$ and
all $a\in\mathcal{L}$  we have
\begin{align*}
P_a^2(x,y) = \sum_{z\in\Omega} P_a(x,z)\, P_a(z,y)
        = \sum_{z\in\Omega_{x,a}} P_a(x,z)\, P_a(z,y)
        &= \sum_{z\in\Omega_{x,a}} P_a(x,z)\, P_a(x,y)\\
        &= P_a(x,y).
\end{align*}
Hence $P_a^2 = P_a$, so $P_a$ is an idempotent matrix.
It follows that $P_a$ is diagonalisable and the only eigenvalues of $P_a$
are 0 and 1.  (See for example~\cite[Section 3.3, Problem 3]{HJ}.) 

Now let $D$ be the diagonal 
$|\Omega|\times|\Omega|$
matrix with diagonal entries
$(D)_{xx} = \sqrt{\pi(x)}$ for $x\in \Omega$.
Define $Q_a = D^{-1} P_a D$ for all $a\in\mathcal{L}$.
Since $P_a$ is reversible with respect
to $\pi$, it follows that $Q_a$ is symmetric.
Furthermore, $Q_a$ is similar to $P_a$ and hence has the
same eigenvalues as $P_a$.  Therefore $Q_a$ is positive semidefinite,
for all $a\in\mathcal{L}$. (Recall that a matrix is positive
semidefinite if it is symmetric and has no negative eigenvalues.)

Now let $Q = \sum_{a\in\mathcal{L}} \rho(a) Q_a$.
Since $Q$ is a nonnegative linear combination of positive
semidefinite matrices, it follows that $Q$ is positive semidefinite.
(See for example~\cite[Observation 7.1.3]{HJ}.) 
Furthermore, by definition we have $P = DQD^{-1}$, so $P$ has the
same eigenvalues as $Q$.  Therefore $P$ has no negative eigenvalues,
as required.
\end{proof}

\subsection{Implications for the mixing time}\label{ss:mixing}

Let $\mathcal{M}$ be an ergodic, reversible Markov chain with
finite state space $\Omega$, transition matrix $P$ and
stationary distribution $\pi$.
The eigenvalues of $\mathcal{M}$ satisfy 
\[ 1 = \lambda_0 > \lambda_1 \geq \lambda_2 \geq \cdots
           \geq \lambda_{N-1} > -1,\]
where $N=|\Omega|$. 
We refer to $\lambda_{N-1}$ as the \emph{smallest eigenvalue}
of $\mathcal{M}$.  
The connection between the mixing time of a Markov chain and its
eigenvalues is well-known (see~\cite[Proposition 1]{sinclair}):
\begin{equation} \tau(\varepsilon) \leq (1-\lambda_\ast)^{-1}\, 
   \ln \frac{1}{\epsilon\, \pi_{\min}}
   \label{mixtime}
\end{equation}
where $\tau(\varepsilon)$ denotes the mixing time of the Markov chain,
$\pi_{\min} = \min_{x\in \Omega} \pi(x)$ and
\[ \lambda_\ast = \max\{ \lambda_1, \, |\lambda_{N-1}|\}.\]
When studying the mixing time of a Markov chain $\mathcal{M}$ using
(\ref{mixtime}),
the approach which has become standard is to make the chain $\mathcal{M}$
\emph{lazy} by replacing $P$ by $(I+P)/2$,
where $I$ denotes the identity matrix.  Then all eigenvalues of the lazy
chain are nonnegative, and only the second-largest eigenvalue must
be investigated.  Clearly if $P$ has no negative eigenvalues 
then $\lambda_\ast = \lambda_1$ 
and it is not necessary to make the chain lazy. 
Our result can be used to quickly verify this for heat-bath chains.

The bound (\ref{mixtime}) underpins many, but not all, methods
of analysing the mixing time of a Markov chain.
Heat-bath chains are often amenable to analysis using the classical
technique of \emph{coupling}, which is not based on (\ref{mixtime}).
(As examples of coupling analyses of heat-bath chains, see~\cite{AMMV,SV}.)
For such chains, the information provided by Lemma~\ref{heat-bath}
does not directly assist in bounding the mixing time.

However, in several applications including~\cite{CDGJM}, a related heat-bath 
Markov chain is analysed using coupling,
and then a comparison argument~\cite{DSC,DGJM} is applied to deduce rapid
mixing of the original heat-bath chain.   Comparison arguments typically
relate the second-largest eigenvalues of the two chains, and hence they
are often applied to lazy Markov chains.
Lemma~\ref{heat-bath} demonstrates that it is unnecessary
to make heat-bath chains lazy when applying the comparison method.

\section{Two applications}\label{s:applications}

In the special case that $\rho$ is the uniform distribution over
$\mathcal{L}$, the equation defining $P$ is
\begin{equation}
\label{Psum}
 P = \frac{1}{L}\, \sum_{a\in\mathcal{L}} P_a.
\end{equation}

\subsection{Application: a single-site heat-bath chain for spin systems}\label{s:spin}

Let $G=(V,E)$ be an arbitrary graph and let $S$ be a finite set
of spins (or colors). Consider a state space $\Omega \subseteq S^V$
and let $\pi\colon\Omega\to(0,1]$ be a probability distribution.
Given $\sigma\in \Omega$, for all $v\in V$ and $k\in S$
we define $\sigma^{v,k}$ by
\[ \sigma^{v,k}(u) = \begin{cases} \sigma(u) & \text{ if $u\neq v$,}\\
          k & \text{ otherwise.}\end{cases}
\]
(So $\sigma^{v,k}$ is obtained from $\sigma$ by replacing the spin at
$v$ by $k$.)
Additionally define, for $\sigma\in\Omega$ and $v\in V$, the set 
$S_v^\sigma=\{k\in S\colon \sigma^{v,k}\in\Omega\}$.
(For spin systems with soft constraints, such as the Ising or Potts models,
 we have $\Omega=S^V$ and $S_v^\sigma=S$ for all $\sigma\in S^V$, $v\in V$.)
The single-site \emph{heat-bath chain} for $\Omega$
is the Markov chain with transition matrix defined by 
\[ P(\sigma,\tau) = \frac{1}{|V|}\,
 \sum_{v\in V} 
\frac{\pi(\tau)}{\sum_{\ell\in S_v^\sigma} \pi(\sigma^{v,\ell})}\, 
\mathds{1}\bigl(\tau = \sigma^{v,\tau(v)}\bigr)
\]
for all $\sigma,\tau\in\Omega$.
This matches the setting of Lemma~\ref{heat-bath} by choosing
$\mathcal{L}=V$ and $L=|V|$, and defining
\[ \Omega_{\sigma,v} = \{ \sigma^{v,k} : k\in S_v^\sigma\}\]
for all $\sigma\in\Omega$ and $v\in V$.
Hence, by Lemma~\ref{heat-bath}, single-site heat-bath chains for
general spin models do not have negative eigenvalues.

\bigskip

The heat-bath chain, which belongs to the family of Glauber dynamics 
(see for example~\cite{M99}), 
has been studied by many authors including~\cite{hayes,LS,MO}.
In several cases, the continuous-time version of this Markov chain 
is considered. One advantage of this approach is that mixing properties 
can be described solely by the second-largest eigenvalue. 
Thus, when translating these results to discrete time, it 
usually remains to bound the smallest eigenvalue of the chain.
The last example shows that for the heat-bath chain, 
the established continuous-time bounds can be used without further analysis. 
In the case of the Potts model this argument was used in the proof 
of~\cite[Theorem~2.10]{U}.

\bigskip

Note that there are other Glauber dynamics, such as the Metropolis chain, 
which are generally not guaranteed to have only nonnegative eigenvalues.

\subsection{Application: a heat-bath chain for contingency tables}\label{s:contingency}

Let $\boldsymbol{r}=(r_1,\ldots, r_m)$ and $\boldsymbol{c}=(c_1,\ldots, c_n)$ 
be two vectors of positive 
integers with the same sum.  A \emph{contingency table} with row sums $r$ and 
column sums $c$ is an $m\times n$ matrix with nonnegative 
integer entries, such that the $i$'th row sum is $r_i$ and the $j$'th
column sum is $c_j$, for $i=1,\ldots, m$ and $j=1,\ldots, n$.
Let $\Omega_{\boldsymbol{r},\boldsymbol{c}}$ denote the set of all 
contingency tables 
with row sums $\boldsymbol{r}$ and column sums $\boldsymbol{c}$.

Dyer and Greenhill~\cite{DG-contingency} proposed a Markov chain for sampling
contingency tables, which we will call the \emph{contingency chain}.  
A transition of the chain is performed as follows:  choose a $2\times 2$ 
subsquare of the current table uniformly at random, then  
replace this $2\times 2$ subsquare by a uniformly chosen $2\times 2$
nonnegative integer matrix with the same row and column sums.
The \emph{lazy} contingency
chain does nothing at each step with probability $\nfrac{1}{2}$, and otherwise
performs a transition as described above.
Cryan et al.~\cite{CDGJM} analysed the lazy contingency chain for a constant
number of rows and proved that it is rapidly mixing.  

To fit the contingency chain into the setting of
(\ref{Psum}),
let $\mathcal{L}$ be the set of all 
positions of $2\times 2$ subsquares, and let 
$L = |\mathcal{L}| = \binom{m}{2}\binom{n}{2}$.  Let $P_a$ be the
transition matrix of the Markov chain which acts only on the 
$2\times 2$ subsquare $a\in \mathcal{L}$.
Then $P_a(x,\cdot)$ is uniform over all contingency tables
$y\in\Omega_{\boldsymbol{r},\boldsymbol{c}}$ which differ from $x$ only
within the $2\times 2$ subsquare $a$.  Hence Lemma~\ref{heat-bath}
applies (with $\pi$ the uniform distribution on 
$\Omega_{\boldsymbol{r},\boldsymbol{c}}$)
and shows that the contingency chain has no negative eigenvalues.

\section{A transfer result for positive semidefiniteness}

The following result on matrices is well known.
(The proof is easy, and omitted.)

\begin{lemma}
\label{transfer}
Consider a state space $\Omega$ with probability distribution 
$\pi\colon\Omega\to(0,1]$ and let $P$ be a transition matrix on $\Omega$.
Let $\Omega'$ be a second state space with 
probability distribution $\mu\colon\Omega'\to (0,1]$. 
Given any  $|\Omega|\times |\Omega'|$ matrix $R$
with rows indexed by $\Omega$ and columns indexed by $\Omega'$,
the adjoint $R^\ast$ of $R$ is defined by
\[ R^*(y,x)=\frac{\pi(x)}{\mu(y)}\, R(x,y)\quad \text{ for all }\,\,
              x\in\Omega,\,\, y\in\Omega'.  \]
Now suppose that 
$P = R\, T R^*$
where $R$ and $T$ satisfy the following conditions:
\begin{itemize}
\item $R$ is a nonnegative $|\Omega|\times|\Omega'|$ matrix 
such that $\pi R = \mu$ and all rows of $R$ sum to one, and
\item $T$ is a positive semidefinite transition matrix on $\Omega'$
which is reversible with respect to $\mu$. 
\end{itemize}
Then $P$ is also positive semidefinite. 
\end{lemma}
Note that we do not assume that $T$ is irreducible. 
(If $R$ is an invertible matrix and $P = RTR^*$ then $P$ and
$T$ are often said to be \emph{congruent}.  But in our applications
$R$ need not be square.)

We now interpret the identity $P = RTR^\ast$ in terms of the
corresponding Markov chains.
Let $\mathcal{M}$ be the Markov chain on $\Omega$ with transition
matrix $P$, and let $\mathcal{M}'$ be the Markov chain on $\Omega'$
with transition matrix $T$.  A transition of $\mathcal{M}$ from current state
$x\in\Omega$ is performed as follows. First, generate an (auxiliary) 
state $x'\in\Omega'$ with 
respect to the probability distribution $R(x,\cdot)$. Then, perform one step 
of the chain $\mathcal{M}'$ from initial state $x'$ 
to obtain $y'\in\Omega'$.  Finally, sample the new state $y\in\Omega$ 
of $\mathcal{M}$ with respect to the distribution $R^*(y',\cdot)$.

Lemma~\ref{transfer} allows us to infer the positive semidefiniteness of
$P$ from the positive semidefiniteness of $T$.  For some applications,
Lemma~\ref{heat-bath} may be used to show that $T$ is positive
semidefinite, while in others we may argue more directly.

As an example, consider the \emph{Swendsen-Wang chain}~\cite{SW} 
for the $q$-state Potts model on a graph $G=(V,E)$.
The state space is $\Omega = \{ 1,2,\ldots, q\}^V$ for some
integer $q\geq 2$. For a fixed constant $\beta \geq 0$, the
stationary distribution of the chain is defined by
\[ \pi(\sigma)=Z^{-1}\, \exp\{\beta |E(\sigma)|\} \quad \text{ for all } \,\,
              \sigma\in\Omega,\] 
where 
\[ E(\sigma)=\{\{u,v\}\in E\colon \sigma(u)=\sigma(v)\bigr\}\] 
denotes the set of monochromatic edges in $\sigma$, and 
$Z$ is the normalizing constant. 
One step of this chain can be described as follows. Given the current state 
$\sigma\in\Omega$, sample a subset $A\subseteq E(\sigma)$ of the monochromatic 
edges such that each edge is included with probability $1-e^{-\beta}$,
with these choices all being independent.
Then, colour each resulting connected component of the subgraph $(V,A)$ with a
new colour chosen from $\{1,\dots,q\}$ uniformly at random, with these
choices all being independent.
(For more details, see for example~\cite{ES,GJ,U}.) 

This fits into the setting of Lemma~\ref{transfer}
if we choose $\pi$, $\Omega$ as above, let
\[ \Omega'=\{(\sigma,A)\colon \sigma\in\Omega,\, A\subseteq E(\sigma)\}\] 
and define $\mu(\sigma,A)=Z^{-1} (e^\beta-1)^{|A|}$ for all $(\sigma,A)\in
\Omega'$. (This $Z$ is the same normalising constant used to define $\pi$.) 
For all $\sigma\in \Omega$ and $(\tau,A)\in\Omega'$, let
\[ R\bigl(\sigma,(\tau,A)\bigr)=e^{-\beta |E(\sigma)|} \,
(e^{\beta}-1)^{|A|}\, \mathds{1}(\sigma=\tau).\] 
Note that with this definition, 
$R^*\bigl((\sigma,A),\tau\bigr)=\mathds{1}(\sigma=\tau)$
for all $\sigma\in\Omega$ and $(\tau,A)\in\Omega'$.
Finally, for all $(\sigma,A),\, (\tau, B)\in\Omega'$, define
\[ T\bigl((\sigma,A),(\tau,B)\bigr)=\mu\bigl((\tau,B)\mid B=A\bigr).\]
It is easy to verify that $\pi R = \mu$ and that
$RTR^*$ equals the transition matrix of the 
Swendsen-Wang chain (see~\cite{ES}).

Now observe that $T$ is idempotent, and hence is positive semidefinite.
We may also conclude this from Lemma~\ref{heat-bath}, since $T$ is a 
(rather trivial) heat-bath Markov chain in the sense of 
Definition~\ref{d:heat-bath}  (where $\mathcal{L} = \{ a\}$ has
a unique element and
setting $\Omega_{x,a}= \Omega$ for all $x\in\Omega$). 
Therefore Lemma~\ref{transfer} shows that the Swendsen-Wang chain 
has no negative eigenvalues, as claimed. 

Another example of a Markov chain which fits the setting of 
Lemma~\ref{transfer}
is the single-bond dynamics for the random-cluster model.
Here Lemma~\ref{heat-bath} is needed in order to prove that
the appropriate matrix $T$ is positive semidefinite.
See~\cite[Section 4.1]{U} for more detail.

\section{A characterisation of heat-bath chains}\label{s:converse}

It follows from the proof of Lemma~\ref{heat-bath} that
any nonnegative linear combination of stochastic idempotent matrices
has only nonnegative eigenvalues.
This leads us to ask whether Lemma~\ref{heat-bath} can be
generalised to a wider class of Markov chains.
To explore this question, we need some more definitions.

We use the symbols  $\boldsymbol{0}$, $\boldsymbol{1}$ to denote any
column vector or row vector with each entry equal to 0, 1
(respectively), of the appropriate size.  We use symbols
$\boldsymbol{a}$, $\boldsymbol{b}$, $\ldots$ to denote column vectors,
and use $\boldsymbol{\alpha}$, $\boldsymbol{\beta}$, $\boldsymbol{\gamma}$, $\ldots$
to denote row vectors.  Unless otherwise noted, the sizes of 
matrices and vectors can be inferred from the context.

A matrix $M$ is called \emph{substochastic} if 
it is nonnegative and
$M\boldsymbol{1} \leq \boldsymbol{1}$. 
A square matrix $M$ is called \emph{permutation similar} to
a matrix $U$ if there is a permutation matrix $A$ such that
$U = A^T MA$.  
This operation corresponds to applying some permutation
to both the rows and the columns of $M$  to obtain $U$.
Since $A^T = A^{-1}$ this is also a matrix
similarity. 
We write $U\cong M$ to show that $U$ and $M$ are permutation
equivalent (or $U\cong_A M$ to specify the permutation matrix $A$).

Note that $M$ is reversible if and only if
there is some nonnegative diagonal matrix $D$ such that 
$MD = DM^T$. In particular, if all diagonal entries of $D$ are positive then
$D^{-1}MD = M^T$.

\begin{remark}
\label{similar}
The equivalence class of $\cong_A$ is closed under multiplication
and the taking of transposes, for all permutation matrices $A$.
Hence if $M$ is stochastic, idempotent or reversible, then so is
any matrix $U$ with $U\cong M$.
\end{remark}

We say that a matrix is an \emph{SI matrix} if it is stochastic
and idempotent.  
An \emph{$r$-SI matrix} will refer to an SI matrix with rank $r$.  
We wish to obtain a characterisation of SI matrices.  
First we consider a generalisation of the 
stochastic case, which we will need later.

\begin{lemma}
Let $M$ be an irreducible, substochastic, idempotent matrix. 
Then $M$ is a 1-SI matrix.  Moreover,
$M = \boldsymbol{1}\boldsymbol{\pi}$, where $\boldsymbol{\pi}$ is a positive vector
with $\boldsymbol{\pi}\boldsymbol{1} = 1$.
\label{med:lem03}
\end{lemma}

\begin{proof}
By idempotence, 0 and 1 are the only possible eigenvalues of $M$.
Since $M$ is irreducible, at least one eigenvalue of $M$ is nonzero.
This implies that $M$ is stochastic, since otherwise $M$ is irreducible
and substochastic, but not stochastic: such matrices have spectral
radius strictly less than one, see~\cite[Corollary 6.2.28]{HJ}. 
Finally, using \cite[Theorem 8.4.4]{HJ}, we obtain that 1 is a simple eigenvalue of $M$. 
This shows that $M$ is a 1-SI matrix and hence is of the form 
$M = \boldsymbol{1}\boldsymbol{\pi}$, where $\boldsymbol{\pi}$ is a positive vector
and $\boldsymbol{\pi}\boldsymbol{1} = 1$.
\end{proof}

This immediately yields the following.

\begin{corollary}\label{med:cor01}
An SI matrix is irreducible if and only if it is a 1-SI matrix.
Furthermore, a matrix $M$ is 1-SI if and only if 
$M = \boldsymbol{1}\boldsymbol{\pi}$
where $\boldsymbol{\pi}$ is positive and $\boldsymbol{\pi}\boldsymbol{1} = 1$.
\end{corollary}

A direct sum of 1-SI matrices is an SI matrix.
However, an SI matrix need not be permutation
equivalent to a direct sum of 1-SI matrices.  
Consider, for example, the matrix
\[ M=\left[\begin{array}{ccc}
\nicefrac12&\nicefrac12 & 0\\
1&0 & 0\\
0&1 & 0\end{array}\right].\]
Clearly $M$ is stochastic, and it is easy to check that $M^2=M$, 
so $M$ is idempotent. But $M$ cannot be permuted to a direct sum of 
1-SI matrices. This is due to the zero column in $M$, which
corresponds to a state which is inaccessible from any state,
including itself. Such a state $y\in\Omega$ is called \emph{ephemeral}
with respect to the Markov chain $\mathcal{M}$ corresponding to $M$.
Ephemeral states can only appear as the initial state of the chain.

The following characterisation of SI matrices depends on the number
of ephemeral states.

\begin{theorem}\label{med:thm01}
Let $M$ be a nonnegative square matrix.
Then $M$ is an SI matrix with exactly $t$ zero columns
if and only if $M\cong U$, where $U$ has the form
\begin{equation}\label{med:eq02}
U=\left[\begin{array}{cccccc}
\boldsymbol{1}\boldsymbol{\pi}_1&0&0&\cdots&0 &0\\
0&\boldsymbol{1}\boldsymbol{\pi}_2&0&\cdots&0 &0\\
0&0&\boldsymbol{1}\boldsymbol{\pi}_3&\cdots&0 &0\\
\vdots&\vdots&\vdots&\ddots&\vdots & \vdots\\
0&0&0&\cdots&\boldsymbol{1}\boldsymbol{\pi}_k &0\\
\boldsymbol{p}_1\boldsymbol{\pi}_1 &
\boldsymbol{p}_2\boldsymbol{\pi}_2 &
\boldsymbol{p}_3\boldsymbol{\pi}_3 &
\cdots &
\boldsymbol{p}_k\boldsymbol{\pi}_k &
0
\end{array}\right]
\end{equation}
for some nonnegative vectors $\boldsymbol{p}_i$ and positive vectors $\boldsymbol{\pi}_i$ 
which satisfy
$\boldsymbol{\pi}_i\boldsymbol{1}=1$
for $i=1,\ldots, k$, and $\sum_{i=1}^k \boldsymbol{p}_i = \boldsymbol{1}$.
(Here all diagonal blocks are square, though not necessarily
of the same size: the last diagonal block has size $t\times t$.)
\end{theorem}

\begin{proof}
Suppose that $M$ is an SI matrix with exactly $t$ zero columns.
Let $M'$ be the matrix obtained from $M$ by removing the $t$ zero
columns as well as the corresponding rows.  Then $M'$ is still
stochastic, and $(M')^2 = M'$, so $M'$ is an SI matrix with 
no zero columns.
It is known~\cite[Section 8.3, Problem 8]{HJ} that $M'\cong U'$, where
\begin{equation}\label{med:eq03}
U'=\left[\begin{array}{ccccc}
A_{11}&A_{12}&A_{13}&\cdots&A_{1k}\\
0&A_{22}&A_{23}&\cdots&A_{2k}\\
0&0&A_{33}&\cdots& A_{3k}\\
\vdots&\vdots&\vdots&\ddots&\vdots\\
0&0&0&\cdots&A_{kk}\end{array}\right]
\end{equation}
such that $A_{ij}\geq 0$ for $1\leq i\leq j\leq k$, and 
$A_{ii}$ is square 
and either irreducible or zero, for $i=1,\ldots, k$. 
Squaring $U'$ gives
\begin{equation}
\label{med:eq04}
(U')^2=\left[\begin{array}{ccccc}
A_{11}^2& B_{12}& B_{13} &\cdots&B_{1k}\\
0&A_{22}^2& B_{23}&\cdots&B_{2k}\\
0&0&A_{33}^2&\cdots& B_{3k}\\
\vdots&\vdots&\vdots&\ddots&\vdots\\
0&0&0&\cdots&A_{kk}^2\end{array}\right],
\end{equation}
for some $B_{ij}\geq 0$ $(i<j\leq k)$.
Hence we have $A_{ii}^2=A_{ii}$ for $i=1,\ldots, k$. 
In particular $U_1=A_{11}$ is idempotent, and $U_1$ is substochastic since 
$U$ is stochastic. Since $U$ has no zero column, $U_1\neq 0$ and hence 
$U_1$ is irreducible. Therefore by Lemma~\ref{med:lem03} it follows that
$U_1$ is stochastic, which implies that $A_{1j}=0$ for $j=2,\ldots, k$. 
Thus $U'=U_1\oplus U''$, where $U_1$ is a 1-SI matrix, and $U''$ is an 
SI matrix with no zero column, or is empty if $k=1$ (in which case $U'=U_1$
is a 1-SI matrix).
By induction, it follows that
$U'= U_1\oplus \cdots \oplus U_k$, where $U_i$ is a 1-SI matrix for $i=1,\ldots, k$.
Applying Corollary~\ref{med:cor01} shows that
$U_i = \boldsymbol{1}\boldsymbol{\pi}_i$, 
where $\boldsymbol{\pi}_i$ is a positive vector which sums to 1, for $i=1,\ldots, k$.

Hence we know that $M\cong U$ where
\[
U =\left[\begin{array}{cccccc}
U_1 & 0& 0&\cdots& 0 & 0\\
0&U_2 & 0&\cdots&0 & 0\\
0&0&U_{3}&\cdots& 0 & 0\\
\vdots&\vdots&\vdots&\ddots&\vdots & \vdots \\
0&0&0&\cdots&U_{k} & 0\\
C_1 & C_2 & C_3 & \cdots & C_k & 0\end{array}\right]
\]
for some nonnegative matrices $C_1,\ldots, C_k$.
Now $U^2=U$, which implies that
\[ C_i = C_iU_i = C_i \boldsymbol{1}\boldsymbol{\pi}_i \]
for $i=1,\ldots, k$.  Let $\boldsymbol{p}_i = C_i\boldsymbol{1}$,
which is a nonnegative vector. 
Then $C_i = \boldsymbol{p}_i\boldsymbol{\pi}_i$ 
and 
\[ \sum_{i=1}^k \boldsymbol{p}_i = \sum_{i=1}^k C_i \boldsymbol{1} = 
  [ \, C_1  \, C_2 \, \cdots \, C_k \, ] \, \boldsymbol{1} = \boldsymbol{1},\]
as $U$ is stochastic.
This completes the proof of the ``only if'' statement.

For the converse, it suffices to assume that $M$ satisfies
(\ref{med:eq02}), by Remark~\ref{similar}.  Then it is not difficult to check that
$M$ is a SI matrix.
\end{proof}

\begin{corollary}\label{med:cor04}
A matrix $M$ is an $r$-SI matrix if and only if $M\cong U$, where $U$ has the 
form~\eqref{med:eq02} with $k=r$.
\end{corollary}

\begin{proof}
By Remark~\ref{similar} it suffices to consider $U$.
If $U$ has the structure given in~\eqref{med:eq02} then $U$ has $k$ groups 
of rows of the form 
$\boldsymbol{\rho}_i=[\,\boldsymbol{0}\,\cdots\,\boldsymbol{\pi}_i\,\cdots\,
  \boldsymbol{0}\,]$ for  $i=1,\ldots, k$. Hence $U$ has rank at least $k$. 
The other rows are of the form
   $[\,\alpha_1\boldsymbol{\pi}_1\,\cdots\,\alpha_k\boldsymbol{\pi}_k\,] = 
  \sum_{i=1}^k\alpha_i\boldsymbol{\rho}_i$, for some nonnegative constants
$\alpha_1,\ldots, \alpha_k$.
 Hence all rows of $U$ are linearly dependent on 
the vectors $\boldsymbol{\rho}_1,\ldots, \boldsymbol{\rho}_k$.  
Thus $U$ has rank exactly $k$, proving that $k=r$.

Conversely, if $M$ is a $r$-SI matrix then $M$ is an SI matrix. 
Theorem~\ref{med:thm01} states that $M\cong U$, where $U$ is given by~\eqref{med:eq02}. 
The argument above then implies that $k=r$.
\end{proof}

We are mostly interested in Markov chains with no ephemeral states
(that is, with no zero columns in their transition matrix),
where the following result will be useful.

\begin{corollary}\label{med:cor03}
Let $M$ be a nonnegative square matrix.
The following are equivalent.
\begin{itemize}
\item[\emph{(i)}]
$M$ is an SI matrix with no zero columns,
\item[\emph{(ii)}] 
$M\cong U$, where $U$ is the direct sum of 1-SI matrices,  
\item[\emph{(iii)}]
$M$ is an SI matrix which is reversible with respect to some positive
distribution: that is, $D^{-1}MD = M^T$ for some diagonal matrix $D$ with
all diagonal entries positive.
\end{itemize}
\end{corollary}

\begin{proof}
That (i) and (ii) are equivalent follows from Theorem~\ref{med:thm01} 
by setting $t=0$.  Next, suppose that $M$ is an SI matrix which 
satisfies $D^{-1}MD = M^T$ for some diagonal matrix $D$ with all diagonal entries 
positive.
This identity implies that if $M$ has a zero column then $M$ also has a zero row, 
contradicting
the fact that $M$ is stochastic.  Hence (iii) implies (i).

Finally, we will prove that (ii) implies (iii).  Note that it suffices
to assume that $M$ is the direct sum of 1-SI matrices, by Remark~\ref{similar}.
Hence we have  $M = U_1\oplus \cdots \oplus U_k$ for some $k\geq 1$, where
$U_i=\boldsymbol{1}\boldsymbol{\pi}_i$ for some positive vector 
$\boldsymbol{\pi}_i$ for $i=1,\ldots, k$. Define the positive vector
$D_i=\operatorname{diag}(\boldsymbol{\pi}_i)$ for $i=1,\ldots, k$, and 
let 
$D= D_1 \oplus \cdots \oplus D_k$. 
Then 
all diagonal entries of $D$ are positive and
\[ DMD^{-1}=\bigoplus_{i=1}^k\, D_iU_iD_i^{-1}
=\bigoplus_{i=1}^k\, (D_i\boldsymbol{1})(\boldsymbol{\pi}_iD_i^{-1})
  =\bigoplus_{i=1}^k\boldsymbol{\pi}_i^T\boldsymbol{1}^T=M^T,
\]
proving that $M$ is reversible with respect to $D$.  
Hence (ii) implies (iii), completing the proof.
\end{proof}

We can now establish the following characterisation of heat-bath Markov chains.

\begin{theorem}
Let $\mathcal{M}$ be a Markov chain on the finite state space $\Omega$,
which is reversible with respect to the probability distribution $\pi:\Omega\to (0,1]$.
Then $\mathcal{M}$ is a heat-bath chain (in the sense of Definition~\ref{d:heat-bath})
if and only if the transition matrix $P$ of $\mathcal{M}$ is a nonnegative linear
combination of nonnegative SI matrices with no zero columns.
\label{characterisation}
\end{theorem}

\begin{proof}
Suppose that $\mathcal{M}$ is a heat-bath matrix.  Then $M$ satisfies (\ref{general})
for some finite set $\mathcal{L}$ and some probability distribution $\rho$ on $\mathcal{L}$.
Recalling (\ref{Pa}) we see that each $P_a$ is stochastic, 
and the proof of Lemma~\ref{heat-bath} shows that each $P_a$ is idempotent.
Finally, note that $P_a$ has no zero columns since $x\in\Omega_{x,a}$ for all $x\in\Omega$,
which implies that $P(x,x)>0$ for all $x\in\Omega$. Hence $P$ has the required form.

For the converse, suppose that $P = \sum_{a\in\mathcal{L}} \rho(a) P_a$  for some
finite set $\mathcal{L}$, where each $P_a$ is a nonnegative
SI matrix with no zero column and $\rho(a)\geq 0$ for all $a\in\mathcal{L}$.
Since $P$ is also stochastic, it follows that $\rho$ is a probability distribution on
$\mathcal{L}$.   Fix $a\in \mathcal{L}$. 
Corollary~\ref{med:cor03} shows that
$P_a$ is permutation-equivalent to a direct sum of 1-SI matrices, which we 
refer to as blocks.  For each $x\in\Omega$, let $\Omega_{x,a}$
be the set of all states which correspond to a row in the block of $P_a$ containing $x$.
(This set is well-defined as it does not depend on the ordering of the blocks.)
It follows that $x\in\Omega_{x,a}$ for all $x\in\Omega$ and $a\in\mathcal{L}$, and
that the sets $\{ \Omega_{x,a} : x\in\Omega\}$ form a partition of $\Omega$, for each
$a\in\mathcal{L}$.  Hence conditions (I), (II) of Section~\ref{s:result} hold.

Now let $B_{x,a}$ be the block of $P_a$ which corresponds to the set $\Omega_{x,a}$.
Then $B_{x,a}$ is a 1-SI matrix, so it equals $\boldsymbol{1}\boldsymbol{\pi}_{x,a}$
for some positive vector $\boldsymbol{\pi}_{x,a}$ which sums to 1.  However, 
$B_{x,a}$ is reversible with respect to the distribution $\pi$ restricted to $\Omega_{x,a}$.
It follows that for all $y,z\in\Omega_{x,a}$ we have
\[ P_a(z,y) = P_a(x,y) = \frac{\pi(y)}{\pi(\Omega_{x,a})},\]
since $B_{x,a}$ has exactly one stationary distribution.
Using the block structure of $P_a$, it follows that $P_a$ satisfies (\ref{Pa}) for all
$a\in\mathcal{L}$.  Therefore $\mathcal{M}$ is a heat-bath chain in the sense of
Definition~\ref{d:heat-bath}, completing the proof.
\end{proof}

\subsection{Chains with finite convergence}

We now investigate a possible generalisation of the notion of an
SI matrix.

Suppose that $M$ is an $n\times n$ stochastic matrix and, for every 
nonnegative vector $\boldsymbol{\alpha}=(\alpha_1,\alpha_2,\ldots,\alpha_n)$
such that $\boldsymbol{\alpha}\boldsymbol{1}=1$, there exists 
a positive integer
$m_{\boldsymbol{\alpha}}$ such that
\begin{equation} \lim_{t\to\infty}\boldsymbol{\alpha} M^t=
  \boldsymbol{\alpha} M^{m_{\boldsymbol{\alpha}}}.
\label{finiteconv}
\end{equation}
(We remark that this condition implies that $M$ is aperiodic.)
Such matrices have been considered before, and correspond to
\emph{chains with finite convergence}~\cite{BG,lindqvist,subelman}.

Write
$\boldsymbol{\alpha} M^{m_{\boldsymbol{\alpha}}}=\boldsymbol{\pi}_{\boldsymbol{\alpha}}$. 
Then $\boldsymbol{\pi}_{\boldsymbol{\alpha}}M=\boldsymbol{\pi}_{\boldsymbol{\alpha}}$, since
\[ \boldsymbol{\pi}_{\boldsymbol{\alpha}}M=\lim_{t\to\infty}\boldsymbol{\alpha} M^{t+1}
   =\lim_{t\to\infty}\boldsymbol{\alpha} M^t=\boldsymbol{\pi}_{\boldsymbol{\alpha}}.\]
Hence the Markov chain corresponding to $M$ converges in a finite number of steps
from any initial distribution.  This generalises the blocks $B_{x,a}$ of the SI matrices $P_a$,
which converge to their stationary distribution after one step.

Let $\boldsymbol{e}_j$ be the $j$th unit (row) vector for $j=1,\ldots, n$.  
(Note, this breaks with our convention of using greek letters for rows and roman letters
for columns.) For ease of
notation, write 
$m_j$ and $\boldsymbol{\pi}_j$ rather than 
$m_{\boldsymbol{e}_j},\,\boldsymbol{\pi}_{\boldsymbol{e}_j}$, for $j=1,\ldots, n$. 
Then
\[ \boldsymbol{\pi}_{\boldsymbol{\alpha}}=\lim_{t\to\infty}\boldsymbol{\alpha} M^t
   = \sum_{j=1}^n\alpha_j\lim_{t\to\infty}\boldsymbol{e}_j M^t
   =\sum_{j=1}^n\alpha_j\boldsymbol{\pi}_j.\]
Let $m=\max_i m_i$, so that
\[ \boldsymbol{e}_jM^m=\boldsymbol{e}_jM^{m_j}M^{m-m_j}=
  \boldsymbol{\pi}_jM^{m-m_j}=\boldsymbol{\pi}_j \quad \text{ for }\,\,
  j=1,\ldots, n.\]
Hence
\[ \boldsymbol{\alpha} M^m = \sum_{j=1}^n\alpha_j\boldsymbol{e}_jM^m = 
\sum_{j=1}^n\alpha_j\boldsymbol{\pi}_j= \boldsymbol{\pi}_{\boldsymbol{\alpha}},\] 
so we may take $m_{\boldsymbol{\alpha}}=m$ for all $\boldsymbol{\alpha}$. Then, for any nonnegative
integer $\delta$,
\[ \boldsymbol{e}_j(M^{m+\delta}-M^m)=\boldsymbol{\pi}_jM^\delta-\boldsymbol{\pi}_j=0\]
for $j=1,\ldots, n$, which implies that $M^{m+\delta}=M^m$.
Taking $\delta=1$ shows that the eigenvalues $\lambda$ of $M$ satisfy $\lambda^m(\lambda-1)=0$. 
Hence the only eigenvalues of $M$ are 0 and 1. 
Taking $\delta=m$ gives $M^{2m}=M^m$, so $M^m$ is idempotent.

If $M$ is also reversible then more is true, as we prove below.
(We remark that the matrices $C_{j}$ which appear in the statement of
Lemma~\ref{l:delta} are more general
than those which arise in the proof of Theorem~\ref{characterisation}.)

\begin{lemma}
Let $M$ be a stochastic matrix with $t$ zero columns.
Suppose that there exists
a positive integer $m$ such that 
\[ M^{m+\delta} = M^m \qquad \text{ for all } \delta\in\mathbb{N}.\]
Then $M\cong U$ for some matrix $U$ with the block structure
\begin{equation}
U =\left[\begin{array}{cccccc}
U_1 & 0& 0&\cdots& 0 & 0\\
0&U_2 & 0&\cdots&0 & 0\\
0&0&U_{3}&\cdots& 0 & 0\\
\vdots&\vdots&\vdots&\ddots&\vdots & \vdots \\
0&0&0&\cdots&U_{k} & 0\\
C_1 & C_2 & C_3 & \cdots & C_k & 0\end{array}\right]
\label{required}
\end{equation}
where the $U_i$ are 1-SI matrices, the $C_i$ are nonnegative,
for $i=1,\ldots, k$, and the last diagonal block has size $t\times t$. 
In particular, if $M$ is reversible with respect to some
positive distribution then $M$ is idempotent and, necessarily, $t=0$.
\label{l:delta}
\end{lemma}

\begin{proof}
By reordering the elements of $\Omega$, we obtain a matrix $U$ 
such that $M\cong U$, where $U$ has the  
block structure
\[
U=\left[\begin{array}{cccccc}
U_1&A_{12}&A_{13}&\cdots&A_{1k} & 0\\
0&U_2&A_{23}&\cdots&A_{2k} & 0 \\
0&0&U_3&\cdots&A_{3k} & 0 \\
\vdots&\vdots&\vdots&\ddots&\vdots & \vdots\\
0&0&0&\cdots&U_k & 0\\
C_1 & C_2 & C_3 & \cdots & C_k & 0
\end{array}\right]
\]
such that $A_{ij}$ is a nonnegative matrix for $1\leq i < j\leq k$,
while $U_{i}$ is square,
irreducible and substochastic and $C_i$ is nonnegative, for $i=1,\ldots, k$.
%

Now each $U_{i}$ is substochastic and irreducible, so $U_i^m$ is substochastic,
irreducible and idempotent. Therefore $U_i^m$ is a 1-SI matrix, by 
Lemma~\ref{med:lem03}.  Hence $U_i$ is stochastic: if the $q$'th row sum
of a substochastic matrix is strictly less than 1, then the 
same is true for any power of that matrix.
It follows that $A_{ij}=0$ for $1\leq i< j\leq k$.

Furthermore, each $U_i$ satisfies $U_i^m(U_i-I)=0$, and hence
has eigenvalue 1 (with multiplicity 1)
with all other eigenvalues zero.  It follows that $U_i$ has rank 1,
and since $U_i$ is stochastic, this implies that 
$U_i = \boldsymbol{1}\boldsymbol{\pi}_i$ for some positive vector 
$\boldsymbol{\pi}_i$, for $i=1,\ldots, k$. 
By Corollary~\ref{med:cor01} it follows that $U_i$ is a 1-SI
matrix for $i=1,\ldots, k$, and (\ref{required}) holds.

Finally, if $M$ is reversible with respect to some positive distribution
then $M$ has no zero columns (that is, $t=0$).
Hence $M$ is permutation equivalent to the direct sum of 1-SI matrices,
and by
Corollary~\ref{med:cor04} it follows that $M$ is idempotent,
as claimed.
\end{proof}

Hence for reversible chains, there is no generalisation to 
Theorem~\ref{characterisation} that can be obtained by replacing
``idempotent'' by some notion of finite convergence, as in (\ref{finiteconv}).

Let $m$ be the smallest integer such that $M^m=M$. 
Then $M$ has complex eigenvalues if $m\geq 4$,
which implies that any matrix satisfying this condition is not reversible.
The case $m=3$ corresponds to periodic Markov chains, which certainly have negative
eigenvalues, and $m=2$ is precisely the idempotence condition.

\end{document}